\documentclass[12pt,english]{article}
\usepackage[T1]{fontenc}
\usepackage[latin9]{inputenc}
\usepackage{amsmath}
\usepackage{amsthm}
\usepackage{amssymb}
\usepackage{enumerate}
\usepackage{color}
\numberwithin{equation}{section}
\makeatletter
\usepackage{amsthm}

\theoremstyle{plain}
\theoremstyle{remark}

\usepackage{fullpage}
\usepackage{makeidx}
\usepackage{amsfonts}
\usepackage{latexsym}
\usepackage{babel}
\usepackage{babel}
\makeatother

\def \1{{\bf 1}}
\def\C{{\mathbb{C}}}
\def\Q{{\mathbb{Q}}}
\def\Z{{\mathbb{Z}}}

\def\N{{\mathbb{N}}}

\theoremstyle{definition}
\newtheorem{lemma}{Lemma}[section]
\newtheorem{theorem}[lemma]{Theorem}
\newtheorem{corollary}[lemma]{Corollary}
\newtheorem{proposition}[lemma]{Proposition}
\newtheorem{definition}[lemma]{Definition}

\newtheorem{remark}[lemma]{Remark}
\usepackage{times}

\usepackage[blocks]{authblk}

\title{Orbifold theory for  vertex  algebras and Galois correspondence}

\author{Chongying Dong \footnote{supported by the Simons foundation  634104 }}
\affil{Department of Mathematics, University of
	California, Santa Cruz, CA 95064 USA}
\author{Li Ren \footnote{partially supported by NSFC grant 12071314}}
\affil{School of Mathematics,  Sichuan University,
	Chengdu 610064 China}
\author{Chao Yang}
\affil{ School of Mathematics,  Sichuan University,
	Chengdu 610064 China }
\begin{document}
\maketitle
\abstract

Let $V$ be a  simple vertex algebra of  countable dimension, $G$ be a finite automorphism group of $V$ and  $\sigma$ be a central element of $G$.
Assume that $\cal S$ is a finite set of inequivalent irreducible $\sigma$-twisted $V$-modules such that  $\cal S$ is invariant under the action of $G$.  Then there is a finite dimensional semisimple associative algebra ${\cal A}_{\alpha}(G,{\cal S})$
for a suitable $2$-cocycle $\alpha$ naturally determined by the $G$-action on ${\cal S}$ such that $({\cal A}_{\alpha}(G,{\cal S}),V^G)$ form a dual pair on the sum $\cal M$ of
$\sigma$-twisted $V$-modules in ${\cal S}$ in the sense that  (1) the actions of ${\cal A}_{\alpha}(G,{\cal S})$ and $V^G$ on $\cal M$ commute,
(2) each irreducible ${\cal A}_{\alpha}(G,{\cal S})$-module appears in $\cal M,$ (3) the multiplicity space of each irreducible ${\cal A}_{\alpha}(G,{\cal S})$-module is an irreducible $V^G$-module, (4) the multiplicitiy spaces of different irreducible ${\cal A}_{\alpha}(G,{\cal S})$-modules are inequivalent $V^G$-modules.
As applications, every irreducible $V$-module is a direct sum of finitely many  irreducible $V^G$-modules and irreducible $V^G$-modules appearing in different $G$-orbits  are inequivalent.  This result generalizes many previous results. We also establish a bijection between subgroups of $G$ and subalgebras of $V$ containing $V^G.$

\section{Introduction}

In this paper we study the orbifold theory for a simple vertex algebra $V$ of countable dimension with a finite automorphism group $G.$ The main goal is to understand $V^G$-modules. Since any $\sigma$-twisted $V$-module $M$ restricts to a $V^G$-module for $\sigma\in G$, the first task is to look for irreducible $V^G$ modules from $M.$ In the present paper we solve this problem if $\sigma$ is a central element of $G.$ More precisely,
for any central element $\sigma$ in $G$ and
any finite set ${\cal S}$ of inequivalent irreducible $\sigma$-twisted modules such that ${\cal S}$ is invariant under the action of $G,$ we establish
a duality result of Schur-Weyl type  for $V^G$ and a semisimple associative algebra $A_{\alpha}(G,{\cal S})$ on the sum of irreducible modules in ${\cal S}.$ In particular, any $\sigma$-twisted module is a direct sum of finitely many irreducible $V^G$-modules.
We also obtain a quantum Galois correspondence for the action of $G$ on $V.$

A systematic study of orbifold theory for a vertex operator algebra initiated in \cite{DM1} and duality result for $(G, V^G)$ on $V$ was established when $G$ is a solvable group. This duality result was extended to any compact Lie group $G$ in \cite{DLM0}.  A Galois correspondence between subgroups of $G$ and vertex operator subalgebras of $V$ containing $V^G$ was derived  in \cite{DM1, HMT, DM2} based on the duality result. The duality result was further extended in \cite{DY} to $\sigma$-twisted modules with $\sigma$ is a central element.  As pointed out in \cite{DY} that  it is better to consider a finite set of inequivalent
irreducible $\sigma$-twisted modules which is G-stable instead of one single module. This
general setting is necessary  for determining the irreducible $V^G$-modules in an irreducible
$\sigma$-twisted $V$-module and identifying irreducible $V^G$-modules from different twisted modules.

Motivated by the theory of associative algebras $A_{g,n}(V)$ for automorphism $g$ of finite order and $n\in \Q$ in \cite{DLM2, DLM3}, an associative algebra $A_{G,n}(V)$ for any finite automorphsim group $G$ and $n\in \Q$ was introduced and studied in \cite{MT}. This leads to
investigating  any finite set of inequivalent irreducible twisted modules which are invariant under the action of $G$ simultaneously. Similar results were also obtained in \cite{DYa, YY} for vertex operator superalgebras.

 If $V$ is rational and $C_2$-cofinite, the results on vertex operator algebra orbifold theory are more fruitful. In this case one can define quantum dimension for a twisted $V$-module and give a full Galois theory using the quantum dimension \cite{DJX}. If $V^G$ is also rational and
 $C_2$-cofinite, then every irreducible $V^G$-module appears in an irreducible $g$-twisted $V$-module \cite{DRX}. These results also hold for vertex operator superalgebas \cite{DRY}.
Also see \cite{CM} on how the rationality and $C_2$-cofiniteness  of $V$ imply the same properties for $V^G$ if $G$ is solvable, and \cite{DNR} on connection between the $V^G$-module category and minimal modular extensions.

All mentioned results on vertex operator algebra orbifold theory heavily depend on the fact that for a vertex operator algebra, an admissible $g$-twisted module  is $\Z_{\geq 0}$-graded
(some kind of highest weight module in the classical sense). Various associative algebras were defined and studied  \cite{Z, DLM1, DLM2, DLM3, MT} for vertex operator algebra to deal with admissible twisted  $V$-modules.  So the techniques and tools for studying vertex operator algebra orbifold theory are not valid for a vertex algebra which itself is not even graded. This explains why investigation of general orbifold theory for vertex algebra is very limited.

But recent study  on Whittaker modules and weak modules  \cite{ALPY1, ALPY2, HY, T} for a vertex operator superalgebra $V$ suggests some results in \cite{DY} on admissible modules above hold for  weak $V$-modules. Our theorems in this paper indicates that all major results in \cite{DY} are true when replacing $V$ by a vertex algebra and admissible modules by any $V$-module (for a vertex algebra we do not have notion of admissible modules). The Schur-Weyl type duality is always the main idea in  studying orbifold theory. The new idea
without using $A_{g,n}(V)$ in this paper is an introduction of generalized Jacobson's density theorem in vertex algebra setting (see Theorem \ref{Dual-Th}). The results  extended  to vertex superalgebra will appear in a separate paper.

Unfortunately, we cannot extend the duality results in \cite{MT} from vertex operator algebra to vertex algebra as we do not know how to deal with all the twisted modules simultaneously  without associative algebras $A_{G,n}(V)$ available. However, we believe that the duality results in \cite{MT} hold for vertex algebra.

This paper is organized as follows. In Section 2, we review basics of vertex algebras.
In section 3, we review some facts on associative algebras, emphasizing the Jacobson's density theorem and its applications to vertex algebras.
In section 4, we introduce dual pairs  in the  vertex algebra setting and prove  key Theorem  {Dual-Th} which connects the Jacobson's density theorem with duality.
In section 5, we prove that  $(\C[G], V^G)$ form a dual pair on $V$.
In section 6, we prove that $(\C^{\alpha_M}[G], V^G)$ form a dual pair on any irreducible $\sigma$-twisted $V$-module $M$ which is $G$-stable where $\sigma$ is a central element of $G.$
In section 7, we prove a most general result that $({\cal A}_{\alpha}(G,{\cal S}) , V^G)$ form a dual pair on $\cal M$
where $\cal M$ is a sum of inequivalent irreducible $\sigma$-twisted modules which are invariant under the action of $G.$

In this paper, we work over the complex field $\C$. For simplicity, $\otimes$ means $\otimes_{\C}.$

\section{Basics}

In this section we review the basics of vertex  algebras (see \cite{B,LL}).

\begin{definition}
A vertex algebra $(V, Y(\ , z) , {\bf 1} )$  is a vector space equipped with
a linear map:
\begin{align*}
		& V \to (\mbox{End}\,V)[[z,z^{-1}]] ,\\
		& v\mapsto Y(v,z)=\sum_{n\in{\Z}}v_nz^{-n-1}\ \ \ \  (v_n\in
		\mbox{End}\,V)\nonumber
\end{align*}
and with a distinguished vector ${\bf 1} \in V$ , satisfying the following conditions: for $u,v\in V$ and $m,n\in \Z$,
\begin{align*} \label{0a4}
		& u_nv=0\ \ \ \ \ {\rm for}\ \  n\ \ {\rm sufficiently\ large};  \\
		& Y({\bf 1},z)=Id_{V};  \\
		& Y(v,z){\bf 1}\in V[[z]]\ \ \ {\rm and}\ \ \
        \lim_{z\to 0}Y(v,z){\bf 1}=v;\\
	\end{align*}
and the Jacobi identity holds:
		\begin{align*}
			& \displaystyle{z^{-1}_0\delta\left(\frac{z_1-z_2}{z_0}\right)
				Y(u,z_1)Y(v,z_2)-z^{-1}_0\delta\left(\frac{z_2-z_1}{-z_0}\right)
				Y(v,z_2)Y(u,z_1)}\\
			& \displaystyle{=z_2^{-1}\delta
				\left(\frac{z_1-z_0}{z_2}\right)
				Y(Y(u,z_0)v,z_2)}
		\end{align*}

\end{definition}

\begin{definition}
An invertible linear transformation $g$ of vertex  algebra  $V$ is called an automorphism if $g\1=\1$ and
$gY(u,z)g^{-1}=Y(gu,z)$ for all $u\in V$. We denote by $\text{Aut}(V)$ the set of all automorphisms of $V$.\\
If  $G \leq \text{Aut}(V),$ then the fixed points
$$V^G=\{v \in V ~| ~ gv=v, g \in G \}$$
is a vertex subalgebra of $V.$
\end{definition}

Let $g$ be an automorphism of $V$ of order $T<\infty.$
For $r=0, \cdots, T-1$,
let $$ V^r=\{v\in V \ | \ gv=e^{-2\pi ir/T}v\}$$
be the eigenspace of $g$ with eigenvalue $e^{-2\pi ir/T}$.
Then $V=\oplus_{r=0}^{T-1}V^r.$
%where $V^r$ is the eigenspace of $g$ with eigenvalue $e^{-2\pi ir/T}.$

\begin{definition}
A $g$-twisted $V$-module $M$ is a vector space equipped with a linear map:
\begin{align*}
	Y_{M}: ~ & V\to\left(\text{End}M\right)\{z^{1/T}, z^{-1/T}\}\\
	& v\mapsto Y_{M}\left(v,z\right)=\sum_{n\in\frac{1}{T}\Z}v_{n}z^{-n-1}\ \left(v_{n}\in\mbox{End}M\right)
	\end{align*}
satisfying  the following for $0\le r\le T-1$, $u\in V^{r}$,
$v\in V$, $w\in M:$
\[
Y_{M}(u,z)=\sum_{n\in-\frac{r}{T}+\Z}u_nz^{-n-1},
\]

\[
u_{l}w=0\ for\ l\gg0,
\]

\[
Y_{M}\left(\mathbf{1},z\right)=Id_{M},
\]

\begin{align*}
z_{0}^{-1}\text{\ensuremath{\delta}}&\left(\frac{z_{1}-z_{2}}{z_{0}}\right)Y_{M}\left(u,z_{1}\right)Y_{M}\left(v,z_{2}\right)-z_{0}^{-1}\delta\left(\frac{z_{2}-z_{1}}{-z_{0}}\right)Y_{M}\left(v,z_{2}\right)Y_{M}\left(u,z_{1}\right)\\
&=z_{2}^{-1}\left(\frac{z_{1}-z_{0}}{z_{2}}\right)^{-r/T}\delta\left(\frac{z_{1}-z_{0}}{z_{2}}\right)Y_{M}\left(Y\left(u,z_{0}\right)v,z_{2}\right).
\end{align*}
\end{definition}

We remark that if $V$ is a vertex operator algebra, the $g$-twisted module in this definition  is called a weak $g$-twisted $V$-module \cite{DLM1}.  For vertex operator algebra we also have the notions of admissible $g$-twisted, ordinary $g$-twisted $V$-modules. But we do not have these notions for vertex algebra as vertex algebra itself does not have a gradation.

If $g=1$ we have the notion of $V$-module for vertex algebra.

The following weak associativity which is useful later follows from the definition (see also \cite{Li2}).

\begin{lemma} (Weak associativity)
Let $M$ be a $g$-twisted $V$-module. Let $u \in V^r$ for some $r$ and $w \in W$.
Let $l$ be a nonnegative integer  such that $u_{n+\frac{r}{T}}w=0$ for any $n\geq l$.
Then we have
$$ (z_0+z_2)^{^{l+\frac{r}{T}}}Y_M(u,z_0+z_2)Y_M(v,z_2)w=(z_2+z_0)^{l+\frac{r}{T}}Y_M(Y(u, z_0)v,z_2)w. $$
\end{lemma}

\begin{lemma} \label{ouov}
Let $M$ be a $g$-twisted $V$-module. Let $X$ be a finite dimensional subspace of $M$.
Let $u^1, \cdots , u^k \in V$ and $n_1, \cdots , n_k \in \frac{1}{T}\Z$.
Then there exist $a^1, \cdots , a^t \in V$ and $m_1, \cdots, m_t \in \frac{1}{T}\Z$ such that
$$u^1_{n_1} \cdots u^k_{n_k}(w)=a^1_{m_1}w+ \cdots + a^t_{m_t}w $$
for any $ w \in X $.
\end{lemma}

\begin{proof}
It is enough to prove that for $u \in V^r, v \in V^s$ and $p,q \in \Z$,
there exist $a^1, \cdots , a^t \in V$ and $m_1, \cdots, m_t \in \frac{r+s}{T}\Z$ such that
$$u_{p+\frac{r}{T}}v_{q+\frac{s}{T}}w=a^1_{m_1}w+ \cdots + a^t_{m_t}w $$ for any $w \in X$.

Since $X$ is finite dimensional, we can take a nonnegative integer $l$ such that $u_{n+\frac{r}{T}}w=0$ for any $n\geq l$ and any $w \in X$.
Consequently, by weak associativity, for any $w \in X$, we have
\begin{align} \label{equ2-1}
(z_0+z_2)^{^{l+\frac{r}{T}}}Y_M(u,z_0+z_2)Y_M(v,z_2)w=(z_2+z_0)^{l+\frac{r}{T}}Y_M(Y(u, z_0)v,z_2)w.
\end{align}
We also note that there exists a nonnegative integer $m$ such that
\begin{align*}
\text{Res}_z z^{p+\frac{s}{T}+n}Y_M(v,z)w=0 ,
\end{align*}
for any  $n>m$ and any  $w \in X$. Hence, for any $i>m$, we have
\begin{align} \label{equ2-2}
\text{Res}_{z_2}
z_0^{p-l-i}z_2^{q+\frac{s}{T}+i}[(z_0+z_2)^{l+\frac{r}{T}}Y_M(u, z_0+z_2)Y_M(v, z_2)w]=0.
\end{align}
For any $w \in X$,  combining (\ref{equ2-1}) and (\ref{equ2-2}), we obtian
\begin{align*}
& u_{p+\frac{r}{T}}v_{q+\frac{s}{T}}w \\
&= \text{Res}_{z_1} \text{Res}_{z_2}z_1^{p+\frac{r}{T}}z_2^{q+\frac{s}{T}}Y_M(u, z_1)Y_M(v, z_2)w\\
& =\text{Res}_{z_0} \text{Res}_{z_1} \text{Res}_{z_2}
z_0^{-1}\delta(\frac{z_1-z_2}{z_0})z_1^{p+\frac{r}{T}}z_2^{q+\frac{s}{T}}Y_M(u, z_1)Y_M(v, z_2)w\\
&=\text{Res}_{z_0} \text{Res}_{z_1} \text{Res}_{z_2}
z_1^{-1}\delta(\frac{z_0+z_2}{z_1})z_1^{p+\frac{r}{T}}z_2^{q+\frac{s}{T}}Y_M(u, z_1)Y_M(v, z_2)w\\
&=\text{Res}_{z_0} \text{Res}_{z_1} \text{Res}_{z_2}
z_1^{-1}\delta(\frac{z_0+z_2}{z_1})(z_0+z_2)^{p+\frac{r}{T}}z_2^{q+\frac{s}{T}}Y_M(u, z_0+z_2)Y_M(v, z_2)w\\
&=\text{Res}_{z_0} \text{Res}_{z_2} (z_0+z_2)^{p+\frac{r}{T}}z_2^{q+\frac{s}{T}}Y_M(u, z_0+z_2)Y_M(v, z_2)w\\
&=\text{Res}_{z_0} \text{Res}_{z_2} (z_0+z_2)^{p-l}z_2^{q+\frac{s}{T}}[(z_0+z_2)^{l+\frac{r}{T}}Y_M(u, z_0+z_2)Y_M(v, z_2)w]\\
&=\text{Res}_{z_0} \text{Res}_{z_2}\sum_{i=0}^m \binom{p-l}{i}
z_0^{p-l-i}z_2^{q+\frac{s}{T}+i}[(z_0+z_2)^{l+\frac{r}{T}}Y_M(u, z_0+z_2)Y_M(v, z_2)w]\\
&=\text{Res}_{z_0} \text{Res}_{z_2}\sum_{i=0}^m \binom{p-l}{i}
z_0^{p-l-i}z_2^{q+\frac{s}{T}+i}[(z_2+z_0)^{l+\frac{r}{T}}Y_M(Y(u, z_0)v,z_2)w]\\
&=\sum_{i=0}^m \sum_{j=0}^\infty \binom{p-l}{i} \binom{l+\frac{r}{T}}{j}(u_{p-l-i+j}v)_{q+\frac{s}{T} + \frac{r}{T}+l+i-j}w.
\end{align*}

If $j$ is large enough, then $u_{p-l-i+j}v=0$ for any $0\leq i \leq m$.
Consequently, the last equation above is a finite sum. Now, the proof is complete.
\end{proof}

\begin{remark}
If $M$ is a $V$-module and $\text{dim}X=1$, Lemma \ref{ouov} was given previously in \cite{LL}.
\end{remark}

For a vertex algebra $V$, let $\mathcal{D}$ be the endomorphism of $V$ defined by
$\mathcal{D}(v)=v_{-2}\1 $  for $v \in V$.  Assume that $g$ is  an automorphism of $V$
of  order $T$. Let $t$ be an indeterminate. Recall from \cite{B,DLM1} that $\mathbb{C}[t^{\frac{1}{T}},t^{-\frac{1}{T}}]$
has the structure of vertex algebra with vertex operator defined by
\[
Y(f(t),z)h(t)=f(t+z)h(t)=(e^{z\frac{d}{dt}}f(t))h(t).
\]
Then the tensor product $\mathcal{L}(V)= V \otimes \mathbb{C}[t^{\frac{1}{T}},t^{-\frac{1}{T}}]$ is also a vertex algebra.
The action of $g$ naturally extends to an automorphism of the tensor
product vertex algebra as follows:
\[
g(u \otimes t^{m})=e^{-2\pi im}(gu \otimes t^{m})
\]
for $u \in V$ and $m \in \frac{1}{T}\Z$.
Denote the subspace of $g$-invariant by $\mathcal{L}(V, g)$, which
is a vertex subalgebra of $\mathcal{L}(V)$.  Moreover, it is clear that
\[
\mathcal{L}(V, g) = \oplus_{r=0}^{T-1} V^{r} \otimes t^{r/T}\mathbb{C}[t,t^{-1}].
\]
Let
\[
V[g]=\mathcal{L}(V,g)/\mathcal{D}\mathcal{L}(V,g).
\]
We will use $v(n)$ to denote the image of $v \otimes t^n$ in $V[g]$ for $v \in V$ and $n \in \frac{1}{T}\Z.$
It is well known that $V[g]$ is a Lie algebra with Lie bracket
\[
[u(n),v(m)]=\sum_{i \geq 0}(u_iv)(m+n-i)
\]
for $u, v \in V$ and $m , n  \in \frac{1}{T}\Z$.

Let $\mathcal{U}(V[g])$ be the universal enveloping algebra of the Lie algebra $V[g]$.
We note that if $M$ is a $g$-twisted $V$-module, then $M$ becomes a $V[g]$-module
such that $u(m)$ acts as $u_{m}$ (see \cite{DLM1}).
Naturally, $M$ is  a $\mathcal{U}(V[g])$-module. If $g=1$, we use  $\mathcal{U}(V)$ to denote $\mathcal{U}(V[g])$.

The following results are obvious.

\begin{lemma} \label{UV}
Let $V$ be a vertex algebra and let $M$ and $N$ be $g$-twisted $V$-modules. Then

\begin{enumerate}[{(1)}]

\item M is a simple $g$-twisted $V$-module if and only if $M$ is a simple $\mathcal{U}(V[g])$-module.

\item $M$ and $N$ are isomorphic $g$-twisted $V$-modules if and only if $M$ and $N$ are isomorphic $\mathcal{U}(V[g])$-modules.

\item (Schur's lemma) If $V$ is  of countable dimensional and $M$ is simple,
then $$\text{End}_{V}(M) \cong \text{End}_{\mathcal{U}(V[g])}(M) \cong \C.$$

\end{enumerate}

\end{lemma}

\section{Associative algebras}

In this section, we review some facts on associative algebras, which are useful later.

\begin{lemma} \label{Lang} \cite[XVII, Theorem 3.2]{Lang}
Let $R$ be a ring and $M$ a semisimple $R$-module. Let $R' = \text{End}_R(M)$.
Let $f \in \text{End}_{R'}(M)$. Let $x_1, \cdots, x_n \in M.$
Then there exists an element $a \in R$ such that $f(x_i)=ax_i$ for $i=1, \cdots, n$.
\end{lemma}

The following well-known Jacobson's density theorem
is an immediate consequence of Lemma \ref{Lang}.

%which is well known follows from Lemma \ref{Lang} immediately.

\begin{corollary} \label{Jac1}
(Jacobson's density theorem I) Let $A$ be an associative algebra over $\C$.
Let $M$ be an  simple $A$-module of countable dimension.
Then for any finite dimensional subspace of $X$ and any $f\in \text{Hom}_\C(X, M)$,
there exists an element $a \in A$ such that $f(x)=ax$ for any $x \in X$.
\end{corollary}

The following  Jacobson's density theorem,  which is a more general version of Corollary \ref{Jac1},  will be used in Section 7.

\begin{corollary} \label{Jac2}
(Jacobson's density theorem II) Let $A$ be an associative algebra over $\C$. Let $M^1, \cdots, M^n$ be inequivalent  simple $A$-modules of countable dimensions.
 Let $X_i $ be finite dimensional subspace of $M^i$ for $i=0, \cdots, n$.
 Then for any $$f\in \text{Hom}_\C(X_1, M^1) \bigoplus \cdots \bigoplus \text{Hom}_\C(X_n, M^n),$$
there exists an element $a \in A$ such that $f(x)=ax$ for any $x \in X_i$ and $i=0, \cdots, 1$.
\end{corollary}
\begin{proof}
Let $A'=\text{End}_A(M^1 \bigoplus \cdots \bigoplus M^n)$.
Since $M^1, \cdots, M^n$ are inequivalent  simple $A$-modules of countable dimensions,
we see that
$$A' \cong \text{End}_A(M^1) \bigoplus \cdots \bigoplus \text{End}_A(M^n) \cong \C^n.$$
%$A' \cong \C^n =C \bigoplus \cdots \bigoplus \C$ .
It is easy to verify that
$$\text{End}_{A'}(M^1 \bigoplus \cdots \bigoplus M^n)= \text{End}_{\C}(M^1) \bigoplus \cdots  \bigoplus \text{End}_{\C}( M^n).$$
Note that any  $f_i\in  \text{Hom}_\C(X_i, M^i)$ can be extended to a linear map
$\widetilde{f_i}\in  \text{End}_\C(M^i).$ So
there exists an element
$$\widetilde{f} \in \text{End}_\C(M_1) \bigoplus \cdots \bigoplus \text{End}_\C(M_n) \cong \text{End}_{A'}(M^1 \bigoplus \cdots \bigoplus M^n)$$
such that $\widetilde{f}=f$ on $X_i$ for $i=1, \cdots, n$.
Now, the result follows from Lemma \ref{Lang}.
\end{proof}

\begin{lemma} \label{not0}
Let $V$ be a simple vertex algebra of countable dimension and let $M$ be a $g$-twisted $V$-module.
If $v^1, \cdots, v^s \in V$ are linearly independent and $w^1, \cdots , w^s \in M$ are not all zero,
then $$Y_M(v^1,z)w^1+ \cdots +Y_M(v^s,z)w^s \neq 0.$$
\end{lemma}

\begin{proof}
Assume that $$Y_M(v^1,z)w^1+ \cdots +Y_M(v^s,z)w^s = 0.$$
%Without loss of generality, we can assume that $v^1, \cdots, v^s \in V^r$ for some $r$.
By weak associativity, for any $r$ and any $u \in V^r$, there exists some $k \in \Z_+$ such that
\begin{align*}
& (z_2+z_0)^{k+\frac{r}{T}}(Y_M(Y(u, z_0)v^1, z_2)w^1 + \cdots + Y_M(Y(u, z_0)v^s, z_2)w^s)\\
&=(z_0+z_2)^{k+\frac{r}{T}}(Y_M(u, z_0+z_2)Y(v^1,z_2)w^1 + \cdots + Y_M(u, z_0+z_2)Y(v^s,z_2)w^s)\\
&=0,
\end{align*}
which implies that
\begin{align*}
Y_M(Y(u, z_0)v^1, z_2)w^1 + \cdots + Y_M(Y(u, z_0)v^s, z_2)w^s=0.
\end{align*}
Hence, for any $u \in V$ and $n \in \Z$, we have
\begin{align} \label{equation1}
Y_M(u_nv^1, z_2)w^1 + \cdots + Y_M(u_nv^s, z_2)w^s=0.
\end{align}

Let $X=\C v_1 \bigoplus  \cdots \bigoplus \C v_n$.
Define a linear mapping $f \in \text{Hom}_\C(X, V)$ by $f(v_1)=v_1$ and $f(v_i)=0$ for any $i \neq 0.$
Since $V$ is a  simple $\mathcal{U}(V)$-module of countable dimension,
by Corollary \ref{Jac1} and Corollary \ref{ouov}, there exist some  $u^1, \cdots, u^p \in V$ and $i_1, \cdots, i_p \in \Z$ such that
$$(u^1_{i_1}+\cdots + u^p_{i_p})v^1=v^1 \ \ \text{and} \ \ (u^1_{i_1}+\cdots + u^p_{i_p})v^i=0, \ \ \text{for} \ \ i \neq 1.$$
By (\ref{equation1}), we obtain  $Y_M(v^1, z_2)w^1=0$.
Similarly, using weak associativity again, we see that
$Y_M(u_nv^1, z_2)w^1=0$ for any $u \in V$ and $n \in \Z$.
Since $V$ is simple, $Y_M(u, z_2)w^1 \equiv 0$ for any $u \in V$. This contradicts to  the fact that $Y_M({\bf 1}, z)w^1=w^1 \neq 0$.
The proof is complete.
\end{proof}

\begin{remark}
When $V$ is a simple vertex algebra and $M$ is an irreducible admissible $V$-module,
Lemma \ref{not0} have been  obtained in \cite{DM1}.

\end{remark}

The following Lemma is well known (see also \cite{Lang} ).

\begin{lemma} \label{faithful1}
Let $G$ be a finite group.
Let $\rho: G \rightarrow \text{End}(M)$ be a finite dimensional faithful representation of $G.$
Then every irreducible representation of $G$ appears in $M^{\otimes n}$ for some $n \geq 0$
(here and below, $M^{\otimes 0}=\C$ is the trivial representation).
\end{lemma}

\begin{corollary} \label{faithful2}
Let $G$ be a finite group. Let $M$ be a faithful representation of $G$ (we do not assume that $M$ is finite dimensional).
Then every irreducible representation of $G$ appears in $M^{\otimes n}$ for some $n \geq 0$.
\end{corollary}
\begin{proof}
Note that $M$ is a semisimple $G$-module. Let $N_1,...,N_k$ be the maximal inequivalent simple $G$-submodules of $M.$ That is, every simple $G$-submodule of $M$ is isomorphic to $N_i$ for some $i.$ Set $N=\sum_{i=1}^kN_i$ which is a finite dimensional faithful submodule of $M.$ By Lemma \ref{faithful1}, every irreducible representation of $G$ appears in $N^{\otimes n}$ for some $n \geq 0.$ The result follows.
\end{proof}

\section{ Duality pairs}

In this section, we introduce dual pairs in the vertex algebra setting.

%, we introduce generalized duality in vertex algebras.
 %We will prove a test theorem of generalized duality: Theorem \ref{Dual-Th}.

\begin{definition}
Let $A$ be an associative algebra and let $V$ be a vertex algebra.
If $M$ is both an $A$-module and a $V$-module such that the actions of $A$ and the actions of $V$ on $M$ commute with
each other, then we call $M$ is an $A \otimes V$-module.\\
The definition of $A \otimes V$-module isomorphism and simple $A \otimes V$-module is obvious.
\end{definition}

\begin{remark}
In the following sections, we will let $A$ be a group algebra or it's generalization and let $V$ be a fixed point vertex  algebra $V^G$.
\end{remark}

Note that for an $A$-module $S$ and a $V$-module $W$, $S \otimes W$ is an $A \otimes V$-module
by defining $a (s \otimes w)=as \otimes w$ and $v_n(s \otimes w)=s \otimes v_nw$ for $a \in A, v \in V, s \in S,   w \in W$ and $n \in \Z$.

%$(a \otimes v_n)(s \otimes w)=as \otimes v_nw$ for $a \in A, s \in S, v \in V,  w \in W$ and $n \in \Z$.

%\begin{lemma} \label{AVmodule}
%If $S$ is a countable dimensional simple $A$-module and $W$ is a simple $V$-module,
%Then  $S \otimes W$ is a simple $A \otimes V$-module.
%\end{lemma}
%\begin{proof}
%The proof is standard. It is enough to prove that every nonzero element in $S \otimes W$ generates whole $S \otimes W$.
%Take any nonzero $s_1 \otimes w_1+ \cdots +s_m \otimes w_m \in S \otimes W$.
%We can assume that $s_1, \cdots, s_m$ are linearly independent and $w_i \neq 0$ for any $i$.
%It follows from Corollary \ref{Jac1} that there exists an element $a \in A$ such that $as_1=s_1$ and $as_i=0$ for any $i \neq 0$.
%Then we have
%$$a(s_1 \otimes w_1+ \cdots +s_m \otimes w_m )=s_1 \otimes w_1.$$
%Since $S$ and $W$ are all simple, it is easy to see that $s_1 \otimes w_1$  generates whole $S \otimes W$, as required.
%\end{proof}

Let $A$ be a finite dimensional semisimple associative algebra.
Let $\Lambda$ be the set of all irreducible  characters of $A$.
For $\lambda \in \Lambda$,  we denote the corresponding irreducible representation by $W_{\lambda}$.

Assume that $M$ is an $A \otimes V$-module.
Let $M^{\lambda}$ be the sum of all  $A$-submodules of $M$  isomorphic to $W_{\lambda}$.
Since the actions of $A$ and the actions of $V$ on $M$ commute with each other,
$M^{\lambda}$ is in fact an $A \otimes V$-submodule of $M$.
Hence, $M$, as an $A \otimes V$-module,  has the following direct sum decomposition:
$$M=\bigoplus_{\lambda \in \Lambda} M^{\lambda}.$$
Let $M_{\lambda}=\text{Hom}_A(W_{\lambda}, M)$ be the multiplicity space of $W_{\lambda}$ in $M$.
Note that
$M_{\lambda}$ has a $V$-module structure by defining
\begin{align} \label{eq4-1}
(v_nf)(w)=v_n(f(w))
\end{align}
for $v \in V, f\in M_{\lambda}, n \in \Z$  and  $ w \in W_{\lambda}.$
We can realize $M_{\lambda}$ as a $V$-submodule of $M$
in the following way: Let $w \in W_{\lambda}$ be a
fixed nonzero
vector. Then we can identify
$\text{Hom}_A(W_{\lambda}, M)$  with the subspace
$$\{f(w) \ | \ f \in  \text{Hom}_A(W_{\lambda}, M) \}$$
of $M^{\lambda}$. Note that the realization depends on the choice of $w$.

%Then, $S_{\lambda} \otimes X_{\lambda}$ is  an $A \otimes V$-module.

Define a linear mapping $\theta: W_{\lambda} \otimes M_{\lambda} \rightarrow M^{\lambda}$ by $\theta(w \otimes f)=f(w)$ for
$w \in W_{\lambda}$ and $f \in M_{\lambda}$.
 Now, we have the following Lemma.

\begin{lemma}
The mapping $\theta$ is an $A \otimes V$-module isomorphism.
\end{lemma}
\begin{proof}
It is well known that $\theta$ is a linear isomorphism (see also \cite[Proposition 4.1.15]{GW}).
It is easy to verify that $\theta$ is an $A \otimes V$-module homomorphism. The proof is complete.
\end{proof}
By the discussion above, $M$, as an $A \otimes V$-module, has the following decomposition:
$$M = \bigoplus_{\lambda \in \Lambda} W_{\lambda} \otimes M_{\lambda}.$$

\begin{definition} \label{def-dual}
Let $A$ be a finite dimensional semisimple associative algebra and let $V$ be a vertex algebra.
If  $M$ is an $A \otimes V$-module such that
\begin{enumerate}[{(1)}]
%\item $X_\lambda$ is nonzero for any $\lambda \in \Lambda(A)$.

\item $M_{\lambda}$ is an irreducible  $V$-module for any $\lambda \in \Lambda,$

\item $M_\lambda \cong M_\mu $ if and only if $\lambda=\mu,$
\end{enumerate}
then we say that the pair $(A, V)$ forms a dual pair on $M.$
\end{definition}

\begin{remark}
Let $M$ be an $A \otimes V$-module.
%Following \cite{JL}
Denote by $\text{Irr}_M(A)$ the set of isomorphism classes of irreducible
$A$-modules $W$ such that $\text{Hom}_A (W, M) \neq 0$
and denote by $\text{Irr}_M(V)$ the set of isomorphism classes of irreducible
$V$-modules $N$ such that $\text{Hom}_V (N, M) \neq 0$.
If $(A, V)$ forms a  dual pair on $M$,
then there is a one-to-one correspondence from $\text{Irr}_M(A)$ to $\text{Irr}_M(V)$.
\end{remark}

\begin{remark}

Assume that $(A,V)$  forms a  dual pair on $M$.
Since $\Lambda$ is a finite set  and $W_\lambda $ is finite dimensional  for any $\lambda \in \Lambda$,
$M$ is a  direct sum of finitely many  irreducible $V$-modules.
\end{remark}

The following result is the key  for the further discussions.

\begin{theorem} \label{Dual-Th}
Let $A$ be a finite dimensional semisimple associative algebra and let $V$ be a vertex algebra.
Let $M$ be an $A \otimes V$-module.  Assume that  for any finite dimensional $A$-submodule $X$ of $M$ and any $f \in \text{Hom}_A(X,M)$,
there exist $v^1, \cdots, v^n \in V$ and $i_1, \cdots, i_n \in \Z$ such that
$$f=v^1_{i_1}+ \cdots + v^n_{i_n}.$$
Then
\begin{enumerate}[{(1)}]
	%\item $X_\lambda$ is nonzero for any $\lambda \in \Lambda(A)$.
	
	\item $M_{\lambda}$ is an irreducible  $V$-module for any $\lambda \in \text{Irr}_M(A),$
	
	\item $M_\lambda \cong M_\mu $ if and only if $\lambda=\mu.$
\end{enumerate}
\end{theorem}
\begin{proof}
(1)  Let  $\lambda \in \text{Irr}_M(A)$ and  $x,y \in M_{\lambda}$ be nonzero.
We need to find  $v^1, \cdots, v^n \in V$ and $i_1, \cdots, i_n \in \Z$ such that $y =(v^1_{i_1}+ \cdots + v^n_{i_n})x$. Consequently, $M_{\lambda}$ is an irreducible $V$-module.

Let $X= W_{\lambda} \otimes x$. Then $X$ is a finite dimensional $A$-submodule of $M$.
Define a mapping $f \in \text{Hom}_\C(X, M)$ by $f(w \otimes x)=w \otimes y$  for $w \in W_\lambda$.
It is claer that $f \in \text{Hom}_A(X, M)$.
Then, there exist $v^1, \cdots, v^n \in V$ and $i_1, \cdots, i_n \in \Z$ such that
$$f=v^1_{i_1}+ \cdots + v^n_{i_n}.$$
Thus  we have
$$w \otimes y =f(w \otimes x)=(v^1_{i_1}+ \cdots + v^n_{i_n})(w \otimes x)=w \otimes (v^1_{i_1}+ \cdots + v^n_{i_n})x,$$
for any $w \in W_\lambda$. This implies that $(v^1_{i_1}+ \cdots + v^n_{i_n})x=y$, as required.\\
% The proof is inspired by the proof of \cite[Theorem 4.2.1]{GW}.

(2) Let  $\lambda, \mu\in \text{Irr}_M(A)$ and $\lambda \neq \mu$. Assume that
$\phi: M_{\lambda}\to M_{\mu}$ be  a $V$-module isomorphism. Let $0\ne x\in M_{\lambda}.$ Then  $0\ne y= \phi(x)\in M_{\mu}.$ Set $X= (W_{\lambda} \otimes x)\oplus (W_{\mu}\otimes y)$ which is an $A$-submodule of $M.$ Define $f\in  \text{Hom}_A(X, M)$
such that $f(w_1\otimes x+w_2\otimes y)=w_2\otimes y$ for $w_1\in W_{\lambda}$ and
$w_2\in W_{\mu}.$ Then there exist $v^1, \cdots, v^n \in V$ and $i_1, \cdots, i_n \in \Z$ such that
$$f=v^1_{i_1}+ \cdots + v^n_{i_n}.$$
Then $(v^1_{i_1}+ \cdots + v^n_{i_n})x=0$ and $(v^1_{i_1}+ \cdots + v^n_{i_n})y=y.$
So
$$0=\phi((v^1_{i_1}+ \cdots + v^n_{i_n})x)=(v^1_{i_1}+ \cdots + v^n_{i_n})\phi(x)=(v^1_{i_1}+ \cdots + v^n_{i_n})y=y,$$
a contradiction. The proof is complete.
\end{proof}

The next result tells us  that the converse of Theorem \ref{Dual-Th} is also true if $V$ is of countable dimension, although we do not  need it  in this paper.

\begin{proposition}
Let $A$ be a finite dimensional semisimple associative algebra and let $V$ be a vertex algebra of  countable dimension.
Let $M$ be an $A \otimes V$-module.
Assume  that  (1)  $M_{\lambda}$ is an irreducible  $V$-module for any $\lambda \in \text{Irr}_M(A),$ (2) $M_\lambda \cong M_\mu $ if and only if $\lambda=\mu.$
Then for any finite dimensional $A$-submodule $X$ of $M$ and  any $f \in \text{Hom}_A(X,M)$,
there exist $v^1, \cdots, v^n \in V$ and $i_1, \cdots, i_n \in \Z$ such that
$$f=v^1_{i_1}+ \cdots + v^n_{i_n}.$$
\end{proposition}
\begin{proof}
	Note that $M=\oplus_{\lambda \in \text{Irr}_M(A)}W_{\lambda}\otimes M_{\lambda}$.  Let $X$ be a finite dimensional $A$-submodule of $M$ and let $f \in \text{Hom}_A(X,M)$.
Note that there exists a finite dimensional subspace $N_{\lambda}$ of $M_\lambda$ for $\lambda \in \text{Irr}_M(A)$ such that
$X=\bigoplus_{\lambda \in \text{Irr}_M(A)} W_{\lambda} \otimes N_{\lambda}$. Moreover, we have
\begin{align*}
\text{Hom}_A(X, M)&=
\text{Hom}_A(\oplus_{\lambda \in \text{Irr}_M(A)} W_{\lambda} \otimes N_{\lambda}, \oplus_{\lambda \in \text{Irr}_M(A)} W_{\lambda} \otimes M_{\lambda})\\
& \cong \oplus_{\lambda \in \text{Irr}_M(A)} \text{Hom}_A(W_\lambda \otimes N_\lambda, W_\lambda \otimes M_\lambda)\\
& \cong \oplus_{\lambda \in \text{Irr}_M(A)} \text{Hom}(N_\lambda, M_\lambda).
\end{align*}
Since $M_{\lambda}, \lambda \in \text{Irr}_M(A)$ are  inequivalent irreducible $\mathcal{U}(V)$-modules of countable dimensions,
it follows from Corollary \ref{ouov} and Corollary \ref{Jac2} that
there  exist $v^1, \cdots, v^n \in V$ and $i_1, \cdots, i_n \in \Z$ such that
$$f=v^1_{i_1}+ \cdots + v^n_{i_n},$$
as desired.
\end{proof}

\section{Duality I}

From now on, we always assume that $V$ is a  simple vertex algebra of countable dimension and $G$ be a finite  automorphism group of $V$.
Note that $V$ is  naturally a faithful $G$-representation.
Moreover, since the actions of $G$ and  $V^G$ on $V$ commute, $V$ is a $\C[G] \otimes V^G$-module. We give a duality result for  the actions of $G$ and  $V^G$ on $V$ in this section.

Let $V[[z_1^{\pm 1}, \cdots, z_n^{\pm 1}]]$ be the $V$-valued formal power series in variables $z_1, \cdots, z_n$.
Then $V[[z_1^{\pm 1}, \cdots, z_n^{\pm 1}]]$ is naturally a $G$-representation such that
$$g(\sum_{i_1, \cdots, i_n \in \Z}v^{i_1, \cdots, i_n}z_1^{i_1}\cdots z_n^{i_n})
=\sum_{i_1, \cdots, i_n \in \Z}(gv^{i_1, \cdots, i_n})z_1^{i_1}\cdots z_n^{i_n},$$
where $g \in G$ and each $v^{i_1, \cdots, i_n} \in V$. Given a $J=(j_1, \cdots, j_n) \in \Z^n,$ we define a projection
$$P_J: V[[z_1^{\pm 1}, \cdots, z_n^{\pm 1}]] \rightarrow V$$
by
$$P_J(\sum_{i_1, \cdots, i_n \in \Z}v^{i_1, \cdots, i_n}z_1^{i_1}\cdots z_n^{i_n})=v^{j_1, \cdots, j_n}.$$
It is clear that for any $J \in \Z^n$, $P_J$ is a $G$-homomorphism.

Let $n \geq 0$. Define a linear mapping
$$\varphi_n: V^{\otimes (n+1)} \rightarrow  V[[z_1^{\pm 1}, \cdots, z_n^{\pm 1}]]$$
by
$$\varphi_n (v^n \otimes \cdots \otimes v^1 \otimes v^0 )=Y(v^n, z_n) \cdots Y(v^1,z_1)v^0.$$
Note that $\varphi_0=Id_V.$

\begin{lemma}\label{l5.1}
The mapping $\varphi_n$ is an injective $G$-homomorphism.
\end{lemma}
\begin{proof}
Since $G \leq \text{Aut}(V)$,  we have
\begin{align*}
\varphi_n(g (v^n \otimes \cdots \otimes v^1 \otimes v^0 ))&=Y(gv^n, z_n) \cdots Y(gv^1,z_1)gv^0\\
&=g(Y(v^n, z_n) \cdots Y(v^1,z_1)v^0)\\
&=g\varphi_n(v^n \otimes \cdots \otimes v^1 \otimes v^0 )
\end{align*}
for $g\in G.$  That is,  $\varphi_n$ is a $G$-homomorphism.

Now we prove that $\varphi_n$ is injective by induction on $n$ (see also \cite{Li1}).
%It is shown in \cite{L} that the map $\varphi_n$ is injective. Here we give a direct proof.
Assume that $v^1 \otimes \alpha_1+ \cdots + v^s \otimes \alpha_s \in \text{Ker}(\varphi_n)$,
where $v^1, \cdots , v^s\in V$ are linearly independent and $\alpha_1, \cdots , \alpha_s \in V^{\otimes n}.$
Then $$Y(v^1,z_n)\varphi_{n-1}(\alpha_1)+ \cdots + Y(v^s, z_n) \varphi_{n-1}(\alpha_s)=0.$$
By Lemma \ref{not0}, we obtain $\varphi_{n-1}(\alpha_1)= \cdots = \varphi_{n-1}(\alpha_s)=0.$
By induction, we have $\alpha_1= \cdots = \alpha_s=0$. Hence $\text{Ker}(\varphi_n)=\{0\}$, as required.
\end{proof}

\begin{corollary} \label{appear-V}
Every irreducible $G$-module  appears in $V$.
\end{corollary}
\begin{proof}
Let $W$ be any irreducible $G$-module.
By Corollary \ref{faithful2}, $W$ appears in $V^{\otimes n}$ for some $n \geq 0$.
If $n=0$, then $W$ is the trivial module. Then,   $W \cong \C {\bf 1} \leq V$ as $G$-modules.
If $n=1$, the proof is complete. Hence we can assume that $n \geq 2$.
Since $\varphi_{n-1}$ is injective, there exists some $J=(j_1, \cdots, j_{n-1}) \in \Z^{n-1}$ such that
$0 \neq P_J \varphi_{n-1}(W) \leq V.$ Hence $W \cong  P_J \varphi_{n-1}(W)$ appears in $V$. The proof is complete.
\end{proof}

\begin{remark}
When $V$ is a vertex operator algebra, Corollary \ref{appear-V} has been obtained in \cite{HMT, DLM0}.
\end{remark}

%\begin{remark}
%If $V$ is a vertex algebra, Corollary \ref{appear-V} have been established in \cite{DM1} preciously.
%\end{remark}

The following lemma is the key to prove the main Theorem \ref{schur1} in this section.

\begin{lemma} \label{f=ov1}
Let $X$ be a finite dimensional $\C G$-submodule of $V$ and let $f$ be any element of $\text{Hom}_G(X,V)$.
Then there exist $v^1, \cdots, v^n \in V^G$ and $i_1, \cdots, i_n \in \Z$ such that
$$f=v^1_{i_1}+ \cdots + v^n_{i_n}.$$
\end{lemma}

\begin{proof}
Note that $V$ is a simple $\mathcal{U}(V)$-module.
Then, by Corollary \ref{Jac1} and Corollary \ref{ouov}, there exist some  $u^1, \cdots, u^n \in V$ and $i_1, \cdots, i_n \in \Z$ such that
$$f=u^1_{i_1}+ \cdots + u^n_{i_n} .$$
Since $f \in \text{Hom}_G(X,V)$, for any $g \in G$, we have
$$f=gfg^{-1}=gu^1_{i_1}g^{-1}+ \cdots + gu^n_{i_n}g^{-1}=(gu^1)_{i_1}+ \cdots + (gu^n)_{i_n}.$$
For $i=1, \cdots, n$, set $v^i= \frac{1}{|G|}\sum_{g \in G}gu^i \in V^G$. Now, we have
$$f= \frac{1}{|G|}\sum_{g \in G} ((gu^1)_{i_1}+ \cdots + (gu^n)_{i_n})=v^1_{i_1}+ \cdots + v^n_{i_n}.$$
This proves the theorem.
\end{proof}

As before let $\Lambda$ be the set of all irreducible  characters of $\C[G].$
For $\lambda \in \Lambda$,  we denote the corresponding irreducible representation by $W_{\lambda}$.

Let $V^{\lambda}$ be the sum of all  $\C [G]$-submodules of $V$  isomorphic to $W_{\lambda}$.
Then $V^{\lambda}$ is a $\C[G] \otimes V^G$-module.
Let $V_{\lambda}=\text{Hom}_G(W_{\lambda}, V)$ be the multiplicity space of $W_{\lambda}$ in $V$.
By the discussion in section 4, $V_\lambda$ is a  $V^G$-module and
hence $W_\lambda \otimes V_\lambda$ is a $\C[G] \otimes V^G$-module.
Moreover, as $\C[G] \otimes V^G$-modules, we have $V^{\lambda} \cong W_\lambda \otimes V_\lambda$.
Hence, $V$, as a $\C[G] \otimes V^G$-module,  has the following direct sum decomposition:
$$V = \bigoplus_{\lambda \in \Lambda} W_{\lambda} \otimes V_{\lambda}.$$
Note that if $\lambda$ is the trivial character, then $V^{\lambda}=V_{\lambda}=V^{G}.$

\begin{theorem} \label{schur1}
The $\C[G] \otimes V^G$-module decomposition:
$$V = \bigoplus_{\lambda \in \Lambda} W_{\lambda} \otimes V_{\lambda}$$
gives a dual pair  $(\C[G], V^G)$ on $V$:
\begin{enumerate}[{(1)}]

\item $V_{\lambda}$ is nonzero for any $\lambda \in \Lambda $.

\item Each $V_{\lambda}$  is an irreducible $V^G$-module.

\item $V_{\lambda}$ and $V_\mu$ are isomorphic $V^G$-modules if and only if $\lambda = \mu$.
\end{enumerate}
In particular, $V$ is a completely reducible $V^G$-module.
\end{theorem}

\begin{proof}

(1) is proved in Corollary \ref{appear-V}.  (2) and (3) follow form Lemma \ref{f=ov1} and Theorem \ref{Dual-Th}. \end{proof}

We remark that Theorem \ref{schur1} generalizes Theorem 2.4 of \cite{DLM0} from vertex operator algebra to vertex algebra.

Using Theorem \ref{schur1} we can give a Galois correspondence between subgroups of $G$ and subalgebras of $V$ containing $V^G,$ generalizing the Galois correspondence from vertex operator algebra \cite{DM1, HMT} to vertex algebra.
\begin{theorem} \label{Galois}
	Let $V$ be a simple vertex algebra of countable dimension and $G$ a finite automorphism group of $V$. Then $H\mapsto V^H$ gives a bijection between subgroups of $G$ and subalgebras of $V$ containing $V^G.$  In particular,  any subalgebra of $V$ containing $V^G$ is simple.
\end{theorem}
\begin{proof}
From Theorem \ref{schur1} we know that for any $H<G,$ $V^H$ is a simple vertex algebra containing $V^G.$ Assume that $H_1,H_2<G$ be two different subgroups of $G.$ Then
$$V^{H_i}=\oplus_{\lambda\in \Lambda}W_{\lambda}^{H_i}\otimes V_{\lambda}$$
for $i=1,2.$ By \cite[Lemma 3.2]{DM1},  there exists
$\lambda\in\Lambda$ such that $W_{\lambda}^{H_1}\ne W_{\lambda}^{H_2}.$ So map $H\mapsto V^H$ is injective.

To prove the map $H\mapsto V^H$ is onto,  we may, and shall,  choose  each $W_{\lambda}$ to be a subspace of $V$ for $\lambda\in \Lambda.$  Let $U$ be a subalgebra of $V$ such that $V^G\subset U\subset V.$
We need to show that there exists $H<G$ such that $U=V^H.$
Note that $U$
 decomposes  of into $V^G$-modules:
 $$U=\oplus_{\lambda\in \Lambda}R_{\lambda}\otimes V_{\lambda}$$
where $R_{\lambda}\subset W_{\lambda}.$ We set
$$W=\bigoplus_{\lambda\in \Lambda}W_{\lambda}, \ \ R=\bigoplus_{\lambda\in \Lambda}R_{\lambda}$$
and consider $W$ as a $G$-submodule of $V.$ Of course $R\subset W.$
The proofs of surjection  given in \cite{HMT, DM2} for vertex operator algebra go through except Lemma \ref{la} below.  The proofs of this Lemma in \cite{HMT, DM2}  for vertex operator algebra are not valid for  vertex algebra.
\end{proof}

 \begin{lemma}\label{la}

Suppose that $\pi$ is a $G$-homomorphism from $W\otimes_{\C}W$ to $W.$ Then
	$\pi(R\otimes R)\subset R.$
\end{lemma}
\begin{proof}

 It is sufficient to prove  that  for any nonzero $\pi: W_{\lambda}\otimes W_{\mu}\to W_{\nu}$
with $\lambda,\mu,\nu$ in $\Lambda,$ $\pi(R_{\lambda}\otimes R_{\mu})\subset R_{\nu}.$	

It is easy to verify that the  following  diagram of $\C[G]$-modules and $\C[G]$-homomorphisms is commutative:
$$\begin{array}[c]{ccc}
\text{Im}(\varphi_1) &  \stackrel{\varphi_1}{\leftarrow} W_\lambda \otimes W_\mu \stackrel{\pi}{\rightarrow} & W_\nu \\
\downarrow\scriptstyle{i} &&  \uparrow\scriptstyle{\rho} \\
V[[z,z^{-1}]] &  \stackrel{=}{\rightarrow} & V[[z,z^{-1}]]
\end{array}$$
where $i$ is inclusion and $\rho$ is some extension of $\pi \circ  \varphi_1^{-1}$ to $V[[z,z^{-1}]]$.

Note that as $\C[G]$-modules, $V[[z,z^{-1}]] \cong \bigoplus_{\lambda \in \Lambda}(W_\gamma \otimes V_{\gamma})[[z,z^{-1}]].$
Since $\rho$ is a $\C[G]$-homomorphism,  $\rho(W_\lambda \otimes V_\lambda[z,z^{-1}])=0$ for $\gamma \neq \nu$.
On the other hand, given an $x \in V_\nu[z,z^{-1}]$, for any $w \in W_\nu$, by Schur's lemma, we have $\rho (w \otimes x)=kw$ for some $k \in \C$.
Hence, we have
\begin{align*}
\rho(U[z,z^{-1}]) & =\rho(\bigoplus_{\lambda \in \Lambda}R_\lambda \otimes V_{\lambda}[z,z^{-1}]) \\
&=\rho(R_\nu \otimes V_{\nu}[z,z^{-1}]) \subseteq R_\nu.
\end{align*}
Note that $\varphi_1(R_\lambda \otimes R_\mu) \subseteq U[[z,z^{-1}]].$
Hence, $\pi(R_\lambda \otimes R_\nu)=\rho ( i ( \varphi_1(R_\lambda \otimes R_\nu))) \subseteq \rho(U[z,z^{-1}])\subseteq R_\nu,$ as required.
\end{proof}
\section{Duality II}

Let $V$ be a countable dimensional simple vertex algebra and let  $G$ be a finite automorphism group of $V$.
Let $Z(G)$ be the center of $G$. In this section and the following sections, we fix an element $\sigma$ of order $T$  in $ Z(G)$.
For $r=0, \cdots, T-1$,
let $ V^r=\{v\in V|\sigma v=e^{-2\pi ir/T}v\}$ be the eigenspace of $\sigma$ associated to the eigenvalue $e^{-2\pi ir/T}$. We will established a duality result for any irreducible $\sigma$-twisted $V$-module in this section.

Let $(M,Y_M(\ , z)$ be an irreducible  $\sigma$-twisted $V$-module.
For $g \in G$,  we define a new irreducible $\sigma$-twisted $V$-module $(M\cdot g, Y_{M\cdot g})$
such that $M \cdot g =M$
as vector space and
$Y_{M \cdot g}(v,z)=Y_M(gv,z)$
for $ v \in V$  \cite{DLM4, DRX}.
%Note that $M \cdot g$ is also an irreducible $V$-module.

In this section, we assume that $M$ is $G$-stable in the sense that for any $g \in G$,
$M$ and $M\cdot g$ are isomorphic $\sigma$-twisted $V$-modules.
Then for $g \in G$ and $v \in V$, there is a linear map $\phi(g): M \rightarrow M$ such that
\begin{align} \label{eq6-1}
\phi(g)Y_M(v,z)\phi(g)^{-1}=Y_M(gv,z).
\end{align}
If $g=1$, we simply take  $\phi(g)=\text{Id}_M.$
The simplicity of $M$ together with Schur's lemma shows that for $g, h \in G$, there exits a nonzero $\alpha_M(g, h)$,
such that
 $$\phi(g)\phi(h)=\alpha_M(g,h)\phi(gh).$$
It is easy to verify that $g\rightarrow \phi(g)$  is a
projection representation of $G$ on $M$ and  $\alpha_M$ is the corresponding $2$-cocycle in $C^2(G, C^*)$ (see \cite{DY}).

Set
$$\C^{\alpha_M}[G]=\bigoplus_{g\in G}\C{\bar g}$$
and define
$${\bar g}{\bar h}=\alpha_M(g ,h){\overline{gh}}$$
for $g , h \in G.$
Then $\C^{\alpha_M}[G]$ is a finite dimensional semisimple associative algebra (see \cite{DY}).
Moreover, $M$ is a $\C^{\alpha_M}[G]$-module such that  ${\bar g}m=\phi(g)m$ for ${\bar g} \in  G$ and $m \in M$.
Since the actions of $G$ and the actions of $V^G$ on $M$ commute with each other, $M$ is a $\C^{\alpha_M}[G] \otimes V^G$-module.

Let $\Lambda_M$ be the set of irreducible characters of $\C^{\alpha_M}[G]$.
For $\lambda \in \Lambda_M$,  we denote the corresponding irreducible representation by $W_{\lambda}$. Let $M^{\lambda}$ be the sum of all  $\C^{\alpha_M}[G]$-submodules of $M$  isomorphic to $W_{\lambda}$.
Then $M^{\lambda}$ is a $\C^{\alpha_M}[G] \otimes V^G$-submodule of $M$.
Let $M_{\lambda}=\text{Hom}_{\C^{\alpha_M}[G]}(W_{\lambda}, M)$ be the multiplicity space of $W_{\lambda}$ in $M$.
Similar to the discussion in Section 4, $M_\lambda$ is a $V^G$-module
and hence $W_\lambda \otimes M_\lambda$ is a $\C^{\alpha_M}[G] \otimes V^G$-module.
Moreover, as $\C^{\alpha_M}[G] \otimes V^G$-modules, we have $M^{\lambda} \cong W_\lambda \otimes M_\lambda$.
Hence, $M$, as a $\C^{\alpha_M}[G] \otimes V^G$-module, has the following direct sum decomposition:
$$M = \bigoplus_{\lambda \in \Lambda_M} W_{\lambda} \otimes M_{\lambda}.$$

The next lemma is a twisted version of Lemma \ref{f=ov1}.
\begin{lemma} \label{f=ov2}
Let $X$ be a finite dimensional $\C^{\alpha_M}[G]$-submodule of $M$ and let $f$ be any element in $\text{Hom}_{\C^{\alpha_M}[G]}(X,M)$.
Then there exist $v^1, \cdots, v^n \in V^G$ and $i_1, \cdots, i_n \in \Z$ such that
$$f=v^1_{i_1}+ \cdots + v^n_{i_n}.$$
\end{lemma}

\begin{proof}
Note that $M$ is a simple $\mathcal{U}(V[\sigma])$-module.
By Corollary \ref{Jac1} and Corollary \ref{ouov}, there exist some  $u^1, \cdots, u^n \in V$ and $i_1, \cdots, i_n \in \frac{1}{T}\Z$ such that
$$f=u^1_{i_1}+ \cdots + u^n_{i_n} .$$
Since $f \in \text{Hom}_{\C^{\alpha_M}[G]}(X,M)$,  for any $g \in G$, by (\ref{eq6-1}), we have
\begin{align} \label{eq6-2}
f & =\phi(g)f\phi(g)^{-1} \notag \\
&=\phi(g)u^1_{i_1}\phi(g)^{-1}+ \cdots + \phi(g)u^n_{i_n}\phi(g)^{-1} \notag \\
&=(gu^1)_{i_1}+ \cdots + (gu^n)_{i_n}.
\end{align}
For $i=1, \cdots, n$, set $v^i= \frac{1}{|G|}\sum_{g \in G}(gu^i) \in V^G$.
By (\ref{eq6-2}), we have
\begin{align*}
f & =\frac{1}{|G|}\sum_{g \in G} ((gu^1)_{i_1}+ \cdots + (gu^n)_{i_n})\\
&=v^1_{i_1}+ \cdots + v^n_{i_n}.
\end{align*}
Note that if  $i_k \notin \Z$ for some $k$ then $v^k_{i_k}=0$ as $v^k \in V^G.$
This proves the theorem.
\end{proof}
To proof Corollary  \ref{abc} below, we need some facts (see \cite{DY,K}).

\begin{enumerate}[{(1)}]

\item If $A$ is a $\C{G}$-module and $B$ is a $\C^{\alpha_M}[G]$-module,
then $A \otimes B$ is a $\C^{\alpha_M}[G]$-module under the action defined by  ${\bar g}(a \otimes b)={ga \otimes {\bar g}b}$
for $a \in A, b\in B$ and $g \in G$.

\item If $A$ is a $\C^{\alpha_M}[G]$-module,
then $A[[z^{\frac{1}{T}}, z^{-\frac{1}{T}}]]$ is also a $\C^{\alpha_M}[G]$-module under the action defined by
$${\bar g}(\sum_{i}a_iz^i)=\sum_{i}({\bar g}a_i)z^i \ \ \text{for} \ \ g \in G \ \ \text{and} \  \ a_i \in A.$$

\item For $n \in \frac{1}{T}\Z$, let $P_{n}$ be the projection from $M[[z^{\frac{1}{T}}, z^{-\frac{1}{T}}]]$ to $Mz^n \cong M$.
Note that $P_n$ is a $\C^{\alpha_M}[G]$-module homomorphism for any $n \in \frac{1}{T}\Z$.

\item Define a linear mapping $\varphi : V \otimes M \rightarrow M[[z^{\frac{1}{T}}, z^{-\frac{1}{T}}]]$
by $\varphi(v \otimes m)=Y_M(v, z)m$ for $v \in V$ and $m \in M$.
By (\ref{eq6-1}) and Lemma  \ref{not0}, $\varphi$ is an injective $\C^{\alpha_M}[G]$-module homomorphism.

\item If $A$ and $B$ are all  $\C^{\alpha_M}[G]$-module, let $A^*=\text{Hom}(A, \C)$.
Then $A^* \otimes B$ is  a $C[G]$-module under the action defined by $g(f \otimes b)=f{\bar g}^{-1} \otimes {\bar g}b$
for $f \in A^*, b \in B$ and $g \in G$.

\item If $S$ is a finite dimensional $\C[G]$-module, $X$ and $Y$ are finite dimensional $\C^{\alpha_M}[G]$-modules, then
\begin{align} \label{iso}
\text{Hom}_{\C[G]}(S, X^* \otimes Y) \cong \text{Hom}_{\C^{\alpha_M}[G]}(S \otimes X, Y).
\end{align}
\end{enumerate}

\begin{remark}
If $\alpha_M \equiv 1$, the proof of  isomorphism (\ref{iso}) above is standard and
follows from the fact that the category of finite dimensional $G$-modules  is rigid (see \cite{EGNO}). For general $\alpha_M $, the proof of (\ref{iso}) is similar.
\end{remark}

\begin{lemma} \label{appearM}
Let $S$ be any simple $\C^{\alpha_M}[G]$-submodule of $V \otimes M$. Then $S$ appears in $M$.
\end{lemma}
\begin{proof}
Since $\varphi$ is  injective,  there exists some $n \in \frac{1}{T}\Z$ such that $0 \neq P_n \varphi(S) \leq M.$
Hence $S \cong P_n \varphi(S) $ is a $\C^{\alpha_M}[G]$-submodule of $M$, as required.
\end{proof}

If $A$ is an associative algebra, let $\text{Irr}(A)$ be the set of the isomorphism classes of simple $A$-modules.

\begin{lemma} \label{5.3}
Let $\Delta \subseteq \text{Irr}(\C^{\alpha_M}[G])$.
Assume that for any $S \in \text{Irr}(\C[G])$ and any $W \in \Delta$,
every irreducible $\C^{\alpha_M}[G]$-submodule of $S \otimes W$ belongs to $\Delta$.
Then $\Delta=\text{Irr}(\C^{\alpha_M}[G])$.
\end{lemma}
\begin{proof}
For any $Y \in \text{Irr}(\C^{\alpha_M}[G])$, $X \in \Delta$, we can take $S \in \text{Irr}(\C[G])$ such that
$$\text{Hom}_{\C[G]}(S, X^* \otimes Y) \neq 0.$$
By (\ref{iso}), we have
$$ \text{Hom}_{\C^{\alpha_M}[G]}(S \otimes X, Y) \cong \text{Hom}_{\C[G]}(S, X^* \otimes Y) \neq 0.$$
Since $\C^{\alpha_M}[G]$ is semisimple, $Y$ appears in $S \otimes X$.
Consequently, $Y \in \Delta$. The proof is complete.
\end{proof}

\begin{corollary} \label{abc}
Every simple $\C^{\alpha_M}[G]$-module appears in $M$.
\end{corollary}

\begin{proof}
Let $\Delta \subseteq \text{Irr}(\C^{\alpha_M}[G])$ be the set of  the isomorphism classes of simple $\C^{\alpha_M}[G]$-modules which appear in $M$.
We claim that  $ \Delta = \text{Irr}(\C^{\alpha_M}[G])$.

Let $S \in \text{Irr}(\C[G])$ and $W \in \Delta .$
By Corollary \ref{appear-V}, $S \otimes W$ appears in $V \otimes M$.
Hence every irreducible submodule of $S \otimes W$ occurs in $M$ by Lemma \ref{appearM}.
Consequently, every irreducible submodule of $S \otimes W$ belongs to $\Delta$.
It follows form Lemma \ref{5.3} that $ \Delta = \text{Irr}(\C^{\alpha_M}[G])$, as required.
\end{proof}

\begin{remark}
When $V$ is a vertex operator algebra, $M$ is an irreducible admissible $V$-module,
Corollary \ref{abc}  have been obtained in \cite{DM1, DY}.
A vertex algebra $V$  might not be $\N$-graded,
the isomorphism  $\psi_s$ used in the proof of Theorem 5.4 of \cite{DY} does not work directly in the current situation.
\end{remark}

\begin{theorem} \label{schur2}
The $\C^{\alpha_M}[G] \otimes V^G$-module decomposition:
$$M = \bigoplus_{\lambda \in \Lambda_M} W_{\lambda} \otimes M_{\lambda}$$
gives a  dual pair $(\C^{\alpha_M}[G], V^G)$ on $M:$

\begin{enumerate}[{(1)}]

\item $M_{\lambda}$ is nonzero for any $\lambda \in \Lambda_M$.

\item Each $M_{\lambda}$  is an irreducible $V^G$-module.

\item $M_{\lambda}$ and $M_\mu$ are isomorphic $V^G$-modules if and only if $\lambda = \mu$.

\end{enumerate}

\end{theorem}
\begin{proof}
(1) is proved in Corollary \ref{abc}.  The proof of (2) and (3) follows form Lemma \ref{f=ov2} and Theorem \ref{Dual-Th}.
\end{proof}

In Theorem \ref{schur2},  if $G=<g>$ is a cyclic group of order $n$ generated by $g,$ we recover the following result  obtained in \cite{ALPY1} previously.
\begin{corollary}
Let $V$ be a vertex operator algebra and  $g$ be an automorphism of order $n$.
Assume that $M$ is an irreducible weak $V$-module such that $M \cdot g \cong M$.
Then  $M$ is the direct sum of $n$ inequivalent irreducible $V^{<g>}$-submodules.
\end{corollary}
\begin{proof}
Note that $G\cong\Z_n.$ It is well known that $H^2(\Z_n, \C^{*})=\{1\}$ where $\C^{*}$ is a trivial $\Z_n$-module (see  \cite{K}). Hence $C^{\alpha}[\Z_n] \cong \C[\Z_n]$ for any $2$-cocycle $\alpha$ in $C^2(G, \C^*)$ and  $C^{\alpha_M}[\Z_n]$ has exactly $n$ inequivalent one-dimensional simple modules.
The result now follows from Theorem \ref{schur2}.\end{proof}

\section{Duality III}

In this section we deal with a finite set of $\sigma$-twisted $V$-modules and extend the duality results in Section 6 further.

We are working on the setting of Section 6 of \cite{DY}.
Let $\cal S$ be a finite set of inequivalent irreducible $\sigma$-twisted $V$-modules. Assume that $\cal S$ is $G$-stable. Then for every $M \in {\cal S}$ and $h \in G$, there exists $N \in {\cal S}$ such that $N \cong M \cdot h$. So there is a linear isomorphism $\phi_{N}(h) : N\rightarrow M$ such that
\begin{align} \label{eq7-1}
\phi_{N}(h)Y_{N}(u,z)
=Y_{M}(hu,z)\phi_{N}(h)
\end{align}
for all $u\in V$.  This implies that $M,N$ are isomorphic as $V^G$-modules.
Since every $\sigma$-twisted module in $\cal S$ is irreducible, there exists $\alpha_{N}(g,h)\in {\mathbb C}^*$
such that
\[\phi_{M}(g)\phi_{N}(h)
=\alpha_{N}(g,h)\phi_{N}(gh).\]
Moreover, for $g,h,k\in G$  we have
\[\alpha_{N}(g,hk) \alpha_{N}(h,k) =
\alpha_{M}(g,h)\alpha_{N}(gh, k).\]
Then
 $${\cal A}_{\alpha}(G,{\cal S})={\mathbb C}[G] \otimes {\mathbb C}{\cal S}
 = \bigoplus_{g \in G, M \in {\cal S}}\C ( g \otimes e(M))$$
is an associative algebra with product
$$g \otimes e(M)\cdot h\otimes e(N)= \alpha_{N}(g,h)gh\otimes e(M \cdot h)e(N)$$
and the identity element
$\sum_{M \in {\cal S}}1\otimes e(M)$ \cite{DY}.

Let ${\cal M}=\bigoplus_{M \in {\cal S}}M$.
We define an action of ${\cal A}_{\alpha}(G,{\cal S})$ on ${\cal M}$ as follows:
for $M, N \in {\cal S},\ w\in N,\ g\in G$ we set
$$(g\otimes e(M)) w  =  \delta_{M, N}\phi_{N}(g)w$$
where $\phi_{N}(g): N\rightarrow N\cdot g^{-1}$.  Under this actions, ${\cal M}$ becomes an ${\cal A}_{\alpha}(G,{\cal S})$-module.
Moreover, by (\ref{eq7-1}), ${\cal M}$  is, in fact, an ${\cal A}_{\alpha}(G,{\cal S}) \otimes V^G$-module.

For $M \in {\cal S}$, let $G_{M}=\{g \in G \ |\ M\cdot g \cong M \}$ be the stabilizer of $M$.
Let ${\cal O}_{M}$ be the orbit of $M$ under the action of $G$.
Let $G=\cup_{j=1}^{k}G_{M}g_j$ be a right coset decomposition with $g_1=1$.
Then $G=\cup_{j=1}^{k}g_j^{-1}G_{M}$ is a left coset decomposition.
Following \cite{DY}, we define several subspaces of ${\cal A}_{\alpha}(G,{\cal S})$ by:
\[
\begin{array}{rcl}
	{\displaystyle S(M)} & = &
	{\displaystyle
		\mbox{Span}\{g\otimes e(M)\ |\ g\in G_{M}\}},\\
	{\displaystyle D(M)} & = &
	{\displaystyle \mbox{Span}\{g\otimes e(M)\ |\ g\in G\}}
	\ \mbox{  and} \\
	{\displaystyle D({\cal O}_{M}}) & = &
	{\displaystyle \mbox{Span}\{g\otimes
		e(M\cdot g_j)\ |\ j=1,\ldots,k, g\in G\}}.
\end{array}
\]
Decompose ${\cal S}$ into a disjoint union of orbits ${\cal S}=\cup_{j\in J}{\cal O}_j$.
Let $M^j$ be a representation of ${\cal O}_j.$
Then ${\cal O}_j=\{M^j \cdot g\ |\ g \in G\}$, and
${\cal A}_{\alpha}(G, {\cal S})=\bigoplus_{j\in J}D({\cal O}_{M^j})$.

The following theorem comes from \cite{DY}.  We give a new proof  for the completeness.

\begin{theorem}\label{AaGS}
With the same notation as above, we have

\begin{enumerate}[{(1)}]
	
\item $S(M)$ is a semisimple associative algebra isomorphic to ${\mathbb C}^{\alpha_{M}}[G_{M}]$,
        where ${\mathbb C}^{\alpha_{M}}[G_{M}]$ is the twisted
		group algebra with $2$-cocycle $\alpha_{M}$.
				
\item The functors  $W \mapsto \mbox{\rm Ind}^{D(M)}_{S(M)}W=D(M) \otimes_{S(M)}W$
and  $X \mapsto (1 \otimes e(M))X$ gives an equivalence between
the category of ${\mathbb C}^{\alpha_{M}}[G_{M}]$-modules and the category of $D({\cal O}_{M})$-modules.
In particular, $D({\cal O}_{M})$ is semisimple
and simple $D({\cal O}_{M})$-modules are precisely equal to
$\mbox{\rm Ind}^{D(M)}_{S(M)}W$
where $W$ ranges over the simple
${\mathbb C}^{\alpha_{M}}[G_{M}]$-modules.
		
\item ${\cal A}_{\alpha}(G, {\cal S}) =\bigoplus_{j\in J}D({\cal O}_{M^j})$ is a direct sum of algebras.
In particular, ${\cal A}_{\alpha}(G, {\cal S})$ is a semisimple associative algebra and
 simple ${\cal A}_{\alpha}(G, {\cal S})$-modules are precisely
$\mbox{\rm Ind}^{D(M^j)}_{S(M^j)}W$, where $W$ ranges over the simple
${\mathbb C}^{\alpha_{M^j}}[G_{M^j}]$-modules
and $j\in J$.
\end{enumerate}
\end{theorem}

\begin{proof}
The proof of (1) is standard (see also \cite{DY,K}). (3) follows from  (2).

We now prove (2).  Let $X$ be any $D({\cal O}_{M})$-module. For $N \in {\cal O}_M$, let $X_N=(1 \otimes e(N))X$.
It is easy to see that $X_N$ is an $S(N)$-module for any $N \in {\cal O}_M$ and $X=\bigoplus_{N \in {\cal O}_M}X_N$.
We also note that if $S$ is an $S(M)$-submodule of $X_M$, then for any $g \in G$,
 $(g \otimes e(M))S$ is an $S(M \cdot g^{-1})$-submodule of $X_{M \cdot g^{-1}}$.
 Moreover, since $\phi_M(g): M \rightarrow M \cdot g^{-1}$ is a linear isomorphism,
 hence $S$ is a simple $S(M)$-module if and only if $(g \otimes e(M))S$ is a simple $S(M \cdot g^{-1})$-module.

Since $S(M)$ is semisimple,  write $X_M=\bigoplus_{i \in I}S_i$, where each $S_i$ is a simple $S(M)$-module.
For $i \in I$,  $D({\cal O}_{M})S_i$ is a $D({\cal O}_{M})$-submodule of $X$ generate by $S_i$.
Note that for any $g \in G $ and $ h \in G_M$, we have
$$(gh \otimes e(M))S_i= \alpha_M(g ,h)^{-1}(g \otimes e(M))(h \otimes e(M))S_i=(g \otimes e(M))S_i.$$
So
\begin{align*}
  D({\cal O}_{M})S_i&=\sum_{g \in G, M \in {\cal O}_M}(g \otimes e(M))S_i \\
&=\sum_{g \in G}(g \otimes e(M))S_i \\
&=\bigoplus_{j=1}^k (g_j^{-1} \otimes e(M))S_i
\end{align*}
Since for any $g \in G$, $(g \otimes e(M))S_i$ is a simple $S(M \cdot g^{-1})$-module,
we conclude  that $D({\cal O}_{M})S_i$ is a simple $D({\cal O}_M)$-module.

Since $\phi_M(g): M \rightarrow M \cdot g^{-1}$ is a linear isomorphism, for any $g \in G$,
$$X_{M \cdot g^{-1}}= \bigoplus_{i \in I}(g \otimes e(M))S_i.$$
 Thus  $X=\bigoplus_{i \in I}D({\cal O}_M)S_i$ and $D({\cal O}_M)$ is semisimple.
 We also conclude that $X$ is generate by $X_M$ and $X_M=0$ if and only if $X=0.$

To complete the proof of (2), we must show that (a)
 if $S$ is an $S(M)$-module, then $(1 \otimes e(M))(D(M) \otimes_{S(M)}S) \cong S,$
(b)
if $X$ is a $D({\cal O}_M)$-module, then $D(M) \otimes_{S(M)}X_M \cong X$.
(a) is clear. For (b) we define a linear mapping
$\rho: D(M) \otimes_{S(M)}X_M \rightarrow  X$ by $ \rho((g \otimes e(M)) \otimes x)= \phi_M(g)(x)$
for $g \in G $ and $x \in X_M$.
It is easy to see that $\rho$ is a  $D({\cal O}_M)$-module homomorphism.
 Since $X$ is generated by $X_M$, then $\rho$ is onto.
It is clear that $\rho: (D(M) \otimes_{S(M)}X_M)_M \cong X_M \rightarrow X_M$ is an isomorphism.
Hence $(\text{ker}(\rho))_M=0$. Consequently, $\text{ker}(\rho)=0$ and $\rho$ is an isomorphism. The proof is complete.
\end{proof}

Note  that if $W$ is a simple  ${\mathbb C}^{\alpha_{M}}[G_{M}]$-module,
then $\text{Hom}_{{\mathbb C}^{\alpha_{M}}[G_{M}]}(W, M)$ is a $V^{G_M}$-module.
In particular, it is a $V^G$-module. It is easy to see that $D(M) \otimes_{S(M)} M$ has a $V^G$-module structure by defining $v_n(x \otimes w)=x \otimes v_nw$
for  $ v \in V^G , n \in \Z$ and $x \otimes w \in D(M) \otimes_{S(M)} M.$
Clearly, the actions of $D({\cal O}_{M})$ and the actions of $V^G$ on $D(M) \otimes_{S(M)} M$ commute with each other.
Hence $D(M) \otimes_{S(M)} M$ is a $D({\cal O}_{M}) \otimes V^G$-module.
By the discussion in section 4, $$\text{Hom}_{D({\cal O}_{M})}(D(M) \otimes_{S(M)} W, D(M) \otimes_{S(M)} M)$$ is a $V^G$-module.

\begin{lemma} \label{123}
Let $W$ be a simple ${\mathbb C}^{\alpha_{M}}[G_{M}]$-module, then
$$\text{Hom}_{{\mathbb C}^{\alpha_{M}}[G_{M}]}(W, M) \ \ \text{and} \ \
 \text{Hom}_{{\cal A}_{\alpha}(G,{\cal S})}(D(M) \otimes _{S(M)} W, {\cal M})$$
are isomorphic $V^G$-modules.
\end{lemma}
\begin{proof}
The proof is divided into several steps.

Claim 1:
$$  \text{Hom}_{\C^{\alpha_{M}}[G_{M}]}(W, M) \ \ \text{and} \ \
 \text{Hom}_{D({\cal O}_{M})}(D(M) \otimes_{S(M)} W, D(M) \otimes_{S(M)} M)$$
 are isomorphic $V^G$-modules.

Define a mapping $$ \theta:  \text{Hom}_{\C^{\alpha_{M}}[G_{M}]}(W, M)
\rightarrow \text{Hom}_{D({\cal O}_{M})}(D(M) \otimes_{S(M)} W, D(M) \otimes_{S(M)} M)$$
by $\theta(f)= \text{Id} \otimes f$ for $f \in \text{Hom}_{\C^{\alpha_{M}}[G_{M}]}(W, M)$.
Since the functor $D(M) \otimes _{S(M)}(-)$ is  an equivalence of categories by Theorem \ref{AaGS}, $\theta$ is a linear isomorphic.
Moreover, for any $v \in V^G, n \in \Z$ and $x \otimes w \in D(M) \otimes_{S(M)} W$, we have
\begin{align*}
&(v_n \theta(f))(x \otimes w)=(v_n(\text{Id} \otimes f))(x \otimes w))\\
&= v_n (( \text{Id} \otimes f)(x \otimes  w))=v_n(x \otimes f(w))=  x \otimes v_nf(w);
\end{align*}
and
$$\theta(v_n f)(x \otimes w)=(\text{Id} \otimes v_nf)(x \otimes w)=x \otimes (v_nf)(w)=x \otimes v_nf(w).$$
Hence $\theta$ is a $V^G$-module isomorphism, as required.

We can also see the claim directly from Theorem \ref{schur2} with $G$ replaced by $G_M.$
We can assume that $W=W_\lambda$ for some $\lambda\in \Lambda_M.$ Then $  \text{Hom}_{\C^{\alpha_{M}}[G_{M}]}(W, M)=M_{\lambda}$ is an irreducible $V^{G_M}$-module and $M=\oplus_{\mu\in \Lambda_M}W_{\mu}\otimes M_{\mu}.$ So
$$D(M) \otimes_{S(M)} M=\oplus_{\mu \in \Lambda_M}(D(M) \otimes_{S(M)} W_{\mu})\otimes M_{\mu}.$$
  From Theorem \ref{AaGS} we know that $D(M) \otimes_{S(M)} W_{\mu}$ are inequivalent simple $D({\cal O}_M)$-modules, the claim follows.

Claim 2:   $D(M) \otimes_{S(M)} M$ and $\bigoplus_{X \in {\cal O}_M} X$ are isomorphic $V^G$-modules .

Define a mapping
$$\rho: D(M) \otimes_{S(M)} M \rightarrow  \bigoplus_{X \in {\cal O}_M} X$$
by
$$\rho((g \otimes e(M)) \otimes x)=\phi_M(g)(x)$$
for $g \in G, x \in M$.
By the proof of Theorem \ref{AaGS} (2), $\rho$ is an ${\cal A}_{\alpha}(G,{\cal S})$-module isomorphism (see also \cite[Proposition 6.2]{DY}).
Since $v_n \phi_M(g)=\phi_M(g) v_n$ for any $g \in G, v \in V^G$ and $n \in \Z$, we see that $\rho$ is a $V^G$-module isomorphism, as required.

Claim 3:
$$\text{Hom}_{{\mathbb C}^{\alpha_{M}}[G_{M}]}(W, M) \ \ \text{and} \ \
\text{Hom}_{{\cal A}_{\alpha}(G,{\cal S})}(D(M) \otimes _{S(M)} W, {\cal M})$$
are isomorphic $V^G$-modules.

Note that if $N \notin {\cal O}_M$, by Theorem \ref{AaGS} (3) and Claim 2, we have
\begin{align*}
\text{Hom}_{{\cal A}_{\alpha}(G,{\cal S})}(D(M) \otimes_{S(M)} W, \bigoplus_{X \in {\cal O}_{N}}X)
&\cong \text{Hom}_{{\cal A}_{\alpha}(G,{\cal S})}(D(M) \otimes_{S(M)} W, D(N) \otimes_{S(N)} N)\\
&=0.
\end{align*}
Hence, as $V^G$-modules, we have
\begin{align*}
& \text{Hom}_{{\mathbb C}^{\alpha_{M}}[G_{M}]}(W, M)\\
&\cong \text{Hom}_{D({\cal O}_{M})}(D(M) \otimes_{S(M)} W, D(M) \otimes_{S(M)} M) \\
&\cong \text{Hom}_{D({\cal O}_{M})}(D(M) \otimes_{S(M)} W, \bigoplus_{X \in {\cal O}_M}X)\\
& \cong \text{Hom}_{{\cal A}_{\alpha}(G,{\cal S})}(D(M) \otimes _{S(M)} W, {\cal M}),
\end{align*}
as expected.
\end{proof}

\begin{lemma} \label{f=ov3}
Let $X$ be a finite dimensional ${\cal A}_{\alpha}(G, {\cal S})$-submodule of ${\cal M}$ and let $f$ be any element of
$\text{Hom}_{{\cal A}_{\alpha}(G, {\cal S})}(X,{\cal M})$.
Then there exist $v^1, \cdots, v^n \in V^G$ and $i_1, \cdots, i_n \in \Z$ such that
$$f=v^1_{i_1}+ \cdots + v^n_{i_n}.$$
\end{lemma}

\begin{proof}
Let $X_M= (1 \otimes e(M))X$ for $M\in {\cal S}.$ Then $X_M \subseteq X \cap M$ and $X= \bigoplus_{M \in {\cal S}}X_M.$ For any $x\in X$ we see that $f((1 \otimes e(M))x)=(1 \otimes e(M))f(x)\in M.$ That is, $f(X_M) \subseteq M$ for any $M \in {\cal S}$.
Consequently, $$f \in \bigoplus_{M \in {\cal S}} \text{Hom}(X_M, M).$$

Note that ${\cal M}$ is a direct sum of  inequivalent irreducible $\mathcal{U}(V[\sigma])$-modules of  countable dimensions.
By Corollary \ref{Jac2} and Corollary \ref{ouov}, there exist some  $u^1, \cdots, u^n \in V$ and $i_1, \cdots, i_n \in \frac{1}{T}\Z$ such that
$$f=u^1_{i_1}+ \cdots + u^n_{i_n} .$$
Since $f \in \text{Hom}_{{\cal A}_{\alpha}(G, {\cal S})}(X,{\cal M})$,
for any $g \in G$ and any $M \in {\cal S}$, we have $(g \otimes e(M \cdot g))f=f(g \otimes e(M \cdot g))$.
Consequently,  $\phi_{M \cdot g}(g)f= f\phi_{M \cdot g}(g)$  for any $g \in G$ and $M \in {\cal S}$.
%Note that the map $\phi_{M \cdot g}(g): M \cdot g \rightarrow M$ is invertible.
%Hence  restricting $f$ on $M$, we have  $f=\phi_{M\cdot g}(g) f \phi_{M\cdot g}(g)^{-1}$.
 For any $g \in G$ and $M \in {\cal S}$,  restricting $f$ to $M$, by (\ref{eq7-1}),  we see that
\begin{align*}
&f=\phi_{M\cdot g}(g) f \phi_{M\cdot g}(g)^{-1}
=\phi_{M\cdot g}(g)u^1_{i_1}\phi_{M\cdot g}(g)^{-1}+ \cdots +\phi_{M \cdot g}(g) u^n_{i_n}\phi_{M \cdot g}(g)^{-1} \\
&=(gu^1)_{i_1}+ \cdots + (gu^n)_{i_n}.
\end{align*}
For $i=1, \cdots, n$, set $v^i= \frac{1}{|G|}\sum_{g \in G}(gu^i) \in V^G$. Now, we have
$$f= \frac{1}{|G|}\sum_{g \in G} ((gu^1)_{i_1}+ \cdots + (gu^n)_{i_n})=v^1_{i_1}+ \cdots + v^n_{i_n}.$$
If  $i_k \notin \Z$ for some $k$, since $v^k \in V^G$, we have $v^k_{i_k}=0$.
Hence, we can assume that $i_k \in \Z$ for all $k$, as required.
\end{proof}

Recall that ${\cal S}=\cup_{j\in J}{\cal O}_j$ and $M^j$ is a representation of ${\cal O}_j.$
For convenience, we let $G_j=G_{M^j},$ $\Lambda_j$ be the set of irreducible characters
of ${\mathbb C}^{\alpha_{M^j}}[G_{M^j}],$ $W_{j, \lambda}$ be the irreducible ${\mathbb C}^{\alpha_{M^j}}[G_{M^j}]$-module with character $\lambda\in \Lambda_j$, $M_{\lambda}^j$ be the multiplicity space of $W_{j,\lambda}$ in $M$
and
$W^{j}_{\lambda}= \mbox{Ind}^{D(M^j)}_{S(M^j)}W_{j, \lambda}.$ Then
$M^j$
has the following decomposition:
\[M^j=\bigoplus_{\lambda\in \Lambda_j}
W_{j,\lambda}\otimes M^j_{\lambda}\]
as ${\mathbb C}^{\alpha_{M^j}}[G_{j}] \otimes V^{G_{j}}$-module and each $ M^j_{\lambda}$ is an irreducible $V^{G_j}$-module by Theorem \ref{schur2}. By the discussion above and Lemma \ref{123}, $\cal M$
has the following decomposition
\[{\cal M}=\bigoplus_{j\in J,\lambda\in\Lambda_j}
W^{j}_{\lambda}\otimes M^j_{\lambda}\]
as ${\cal A}_{\alpha}(G, {\cal S})\otimes V^G$-module.

Finally, we have
\begin{theorem}\label{t5.6}\label{schur3}
The  ${\cal A}_{\alpha}(G,{\cal S})\otimes V^G$-module decomposition
	\[{\cal M}=\bigoplus_{j\in J,\lambda\in\Lambda_j}
	W^{j}_{\lambda}\otimes M^j_{\lambda}\]
gives a  dual pair $({\cal A}_{\alpha}(G,{\cal S}), V^G)$ on ${\cal M}:$
\begin{enumerate}[{(1)}]
\item $M^{j}_{\lambda}$ is nonzero for any $j\in J$ and $ \lambda\in\Lambda_j$.

\item  Each $M^{j}_{\lambda}$ is an irreducible $V^G$-module.
		
\item $M^{j_1}_{\lambda_1}$ and $M^{j_2}_{\lambda_2}$ are isomorphic
$V^G$-modules if and only if $j_1=j_2$ and $\lambda_1=\lambda_2$.
\end{enumerate}	
In particular, all irreducible $\sigma$-twisted $V$ module is a direct sum of finitely many
irreducible $V^G$-modules.
\end{theorem}

\begin{proof}
(1) follows from Theorem \ref{schur2} (2) immediately.  (2) and (3) follow form Lemma \ref{f=ov3} and Theorem \ref{Dual-Th}.
\end{proof}

The following Corollary recovers a recent result given in  \cite{ALPY2}.
\begin{corollary}
Let $M$ be an irreducible  $\sigma$-twisted $V$-module.
Then $M$ is an irreducible $V^G$-module if and only if $G_M=\{1\}.$
\end{corollary}

\begin{proof}
Let ${\cal S}= \{M\cdot g \ | \ g \in G \}.$  Note that ${\cal A}_{\alpha}(G,{\cal S})=D({\cal O}_M)$. It follows from Theorem  \ref{schur3} that $M$ is irreducible $V^G$-module if and only if $\C^{\alpha_M}[G_M] \cong \C $ which is clearly equivalent to that $G_M=\{1\}.$
\end{proof}

\end{document}